\documentclass[11pt,reqno,a4paper]{amsart}

\usepackage{color}
\usepackage{amssymb,epsfig,graphics}
\usepackage{dsfont}

\newtheorem{theorem}{Theorem}[section]

\newtheorem{lemma}[theorem]{Lemma}
\newtheorem{sub-lemma}[theorem]{Sub-Lemma}

\def\E{\mathds{E}}
\def\N{\mathds{N}}
\def\Z{\mathds{Z}}

\def\P{\mathds{P}}

\def\R{\mathds{R}}

\def\Z{\mathds{Z}}
\def\D{\mathcal{D}}

\begin{document}

\title[Empirical processes for RWRS]{Empirical processes for recurrent and transient random walks in random scenery}
\author[Guillotin-Plantard]{Nadine Guillotin-Plantard}
\address{Institut Camille Jordan, CNRS UMR 5208, Universit\'e de Lyon, Universit\'e Lyon 1, 43, Boulevard du 11 novembre 1918, 69622 Villeurbanne, France.}
\email{nadine.guillotin@univ-lyon1.fr}

\author[P\`{e}ne]{Fran\c{c}oise P\`{e}ne}
\address{Univ Brest, Universit\'e de Brest, IUF, LMBA, UMR CNRS 6205, 29238 Brest cedex, France}
\email{francoise.pene@univ-brest.fr}

\author[Wendler]{Martin Wendler}
\address{Institut f\"ur mathematische Stochastik, Otto-von-Guericke-Universit\"at, 39106 Magdeburg, Germany}
\email{martin.wendler@ovgu.de}


\begin{abstract}
In this paper, we are interested in the asymptotic behaviour of the sequence 
of processes $(W_n(s,t))_{s,t\in[0,1]}$ with
\begin{equation*}
W_n(s,t):=\sum_{k=1}^{\lfloor nt\rfloor}\big(\mathds{1}_{\{\xi_{S_k}\leq s\}}-s\big)
\end{equation*}
where $(\xi_x, x\in\mathbb \Z^d)$ is a sequence of independent random variables uniformly distributed on $[0,1]$ and $(S_n)_{n\in\N}$ is a random walk evolving in $\Z^d$, independent of the $\xi$'s. In \cite{wend16}, the case where $(S_n)_{n\in\N}$ is a recurrent random walk in $\Z$ such that $(n^{-\frac 1\alpha}S_n)_{n\geq 1}$ converges in distribution to a stable distribution of index $\alpha$, with $\alpha\in(1,2]$, has been investigated. Here, we consider the cases where $(S_n)_{n\in\N}$ is either~:\begin{itemize}
\item[(a)] a transient random walk in $\Z^d$,
\item[(b)] a recurrent random walk in $\Z^d$ such that $(n^{-\frac 1d}S_n)_{n\geq 1}$ converges in distribution to a stable distribution of index $d\in\{1,2\}$.
\end{itemize}
\end{abstract}

\date{\today}
\maketitle

\bibliographystyle{plain}

\section{Introduction and Main Results}

The sequential empirical process has been studied under various assumptions, starting with M\"uller \cite{mull70} under independence. In this paper, we will study the asymptotic behaviour of the sequence of processes $(W_n(s,t))_{s,t\in[0,1]}$ with
\begin{equation}\label{defineprocess}
W_n(s,t):=\sum_{k=1}^{\lfloor nt\rfloor}\big(\mathds{1}_{\{\xi_{S_k}\leq s\}}-s\big)
\end{equation}
where $(S_n)_n$ is either~:
\begin{itemize}
\item[(a)] a transient random walk in $\Z^d$,
\item[(b)] a recurrent random walk in $\Z^d$ such that $(n^{-\frac 1d}S_n)_{n\geq 1}$ converges in distribution to a stable distribution of index $d\in\{1,2\}$
\end{itemize}
and $(\xi_x)_{x\in \Z^d}$ is a sequence or random field of independent random variables uniformly distributed on $[0,1]$, independent of $(S_n)_n$.
The process $(\xi_{S_k})_{k\geq 1}$ can be viewed as the increments of a random walk in random scenery (RWRS, in short) 
$(Z_n)_{n\geq 1}$. In other words, 
\begin{equation*}
\forall n\geq 1, \quad Z_n = \sum_{k=1}^n \xi_{S_k}.
\end{equation*}
To simplify we will assume that the random walk is aperiodic in the sense of Spitzer \cite{spi76},
which amounts to requiring that $\varphi(u) = 1$ if and only if $u \in 2 \pi \Z^d$,
where $\varphi$ is the characteristic function of $S_1$.

RWRS was first introduced in dimension one by Kesten and Spitzer \cite{kest79} and Borodin \cite{MR543530,MR535455} in order to construct new self-similar stochastic processes.
For $d=1$, Kesten and Spitzer \cite{kest79} proved that when 
the random walk and the random scenery belong to the domains of attraction of different stable laws
of indices $1<\alpha\leq 2$ and $0<\beta\leq 2$, respectively,
then there exists $\delta>\frac{1}{2}$ such that
$\big(n^{-\delta}Z_{[nt]}\big)_{t\geq 0}$ converges weakly as $n\rightarrow
\infty$ to a continuous $\delta$-self-similar process with stationary increments,
$\delta$ being related to $\alpha$ and $\beta$
 by $\delta=1-\alpha^{-1}+(\alpha\beta)^{-1}$. The limiting process can be seen as a mixture of  $\beta$-stable processes, but it is not a stable process.
When $0<\alpha<1$ and for arbitrary $\beta$, the sequence $\big(n^{-\frac{1}{\beta}}Z_{[nt]}\big)_{t\geq 0}$ converges
weakly, as $n\rightarrow \infty$, to a stable process with index $\beta$ (see \cite{cast13}). Bolthausen \cite{bolt89} (see also \cite{degli11}) gave a method to solve the case $\alpha=1$ and $\beta=2$ and especially, he proved
that when $(S_{n})_{n\in \mathbb{N}}$ is a recurrent $\mathbb{Z}^{2}$-random
walk, the sequence $\big((n\log n)^{-\frac{1}{2}}Z_{[nt]}\big)_{t\geq 0}$ satisfies a functional
central limit theorem. More recently, the case $d=\alpha\in\{1,2\}$ and $\beta\in (0,2)$ was solved in \cite{cast13},  the authors prove that the sequence 
$\big(n^{-1/\beta} (\log n)^{1/\beta -1}   Z_{[nt]} \big)_{t\geq 0}$ converges weakly  to a stable process with index $\beta$. Finally for any arbitrary transient $\mathbb{Z}^{d}$-random walk, it can be shown that the sequence $(n^{-\frac{1}{2}}Z_{n})_n$ is asymptotically normal (see for instance \cite{spi76} page 53). 

Far from being exhaustive, we can cite strong approximation results and laws of the iterated logarithm \cite{csa83,csa99,KL98}, limit theorems for correlated sceneries or walks \cite{GuiPri1,GuiPri2,coh09}, large and moderate deviations results \cite{ass07,Nina07, cast01, cast04}, ergodic and mixing properties (see the survey \cite{denhol06}). 

The problem we investigate in the present paper has already been studied in \cite{wend16} in the case where $(S_n)_{n\in\N}$ is a recurrent random walk in $\Z$ such that $(n^{-\frac 1\alpha}S_n)_{n\geq 1}$ converges in distribution to a stable distribution of index $\alpha$, with $\alpha\in(1,2]$.
In \cite{wend16}, the limit process differed from the limit process of the sequential empirical process of independent random variables. We will show that in other cases, we obtain the classical limit process for the sequential empirical process (as under independence).
Let us recall that a Kiefer-M\"uller process $W:=\big(W(s,t)\big)_{s,t\in[0,1]}$ is a centered two-parameter Gaussian process with covariances
\begin{equation*}
\E\left[W(s,t)W(s',t')\right]=t\wedge t'(s\wedge s'-ss').
\end{equation*}
The study of this sequential process has been initiated independently by M\"uller in \cite{mull70} and by Kiefer in \cite{Kiefer72}.
\begin{theorem}\label{theo1}
Assume that one of the following assumptions holds
\begin{itemize}
\item[(a)] $(S_n)_n$ is a transient random walk on $\Z^d$,
with $d\in \N^*$, $a_n:=\sqrt{n}$,
\item[(b1)] $d=1$, $(S_n/n)_n$ converges in distribution to a random variable with characteristic function $t\mapsto \exp(-A|t|)$
with $A>0$, $a_n:=\sqrt{n\log n}$,
\item[(b2)] $d=2$, the random walk increment $S_1$ is centered and square integrable with
invertible covariance matrix $\Sigma$ and $A:=2\sqrt{\det\Sigma}$,
 $a_n:=\sqrt{n\log n}$.
\end{itemize}
Then $\left(a_n^{-1} (W_n(s,t))_{s,t\in[0,1]}\right)_n$ converges in distribution 
in the Skorohod space
 $\D\left([0,1]^2,\R\right)$
of c\`adl\` ag functions
to $\left(\sqrt{c}W(s,t)\right)_{s,t\in[0,1]}$, where $W$ is a Kiefer-M\"uller process and
\begin{itemize}
\item $\displaystyle c=1+2\sum_{k=1}^\infty \P(S_k=0)$ in case (a).
\item $\displaystyle c=\frac{2}{\pi A}$ in cases (b1) and (b2).
\end{itemize}
\end{theorem}
Note that the limit process is the same as under independence
of the observables $\xi_{S_n}$ (e.g. when $S_n=n$),
even if the norming is different in the cases (b1) and (b2). In contrast, for intermittent maps, Dedecker, Dehling and Taqqu \cite{dedecker} have shown that the same $\sqrt{n\log n}$ norming is needed, but the limit process behaves drastically different and is degenerate: As in the long range dependent case (see Dehling, Taqqu \cite{dehling}), the limit is degenerate, meaning that it can be expressed as $(c(s)Z(t))_{s,t\in[0,1]}$, where $c(s)$ is a deterministic function and $(Z(t))_{t\in[0,1]}$ is a stochastic process.
Note that even under short range dependence, the limit might be distorted, see Berkes and Philipp \cite{berkes}. In the case of a random walk in random scenery with $\alpha>1=d$, a much stronger norming is needed, but the limit is also not degenerate.

If we consider a random walk in random scenery $(X_{S_k})_{k\in\N}$ with random variables $(X_x)_{x\in \Z^d}$ not uniformly distributed on the interval $[0,1]$, the limit distribution of the sequential empirical process can still be deduced from Theorem \ref{theo1}. Let $F_{X}$ be the distribution function of the random variables $X_x$. Furthermore, let $(\xi_x)_{x\in \Z^d}$ be independent and uniformly distributed on $[0,1]$ as before. Then the sequential empirical process $(V_{s,t})_{s\in\R,t\in[0,1]}$ with
\begin{equation*}
V_n(s,t):=\sum_{k=1}^{\lfloor nt\rfloor}\big(\mathds{1}_{\{X_{S_k}\leq s\}}-F_X(s)\big)
\end{equation*}
has the same distribution as $(W_n(F_X(s),t))_{s\in\R,t\in[0,1]}$ with $W_n$ defined in (\ref{defineprocess}). To see this,  define the quantile function
\begin{equation*}
F_{X}^{-1}(s):=\inf\left\{x\big|F_{X}(x)\geq s\right\}.
\end{equation*}
It is well known that the quantile function satisfies $F_{X}^{-1}(s)\leq s'$ if and only if $s\leq F_X(s')$ (see e.g. the book of Billingsley \cite{bill99}, chapter 14). So 
\begin{equation*}
W_n(F_X(s),t)=\sum_{k=1}^{\lfloor nt\rfloor}\big(\mathds{1}_{\{\xi_{S_k}\leq F_X(s)\}}-F_X(s)\big)=\sum_{k=1}^{\lfloor nt\rfloor}\big(\mathds{1}_{\{F_X^{-1}(\xi_{S_k})\leq s\}}-F_X(s)\big)
\end{equation*}
and $\P(F_X^{-1}(\xi_{x})\leq s)=\P(\xi_x \leq F_X(s))=F_X(s)=\P(X\leq s)$, so the random variables $(F^{-1}_X(\xi_{x}))_{x\in\Z^d}$ are independent with distribution function $F_X$. So it suffices to study the case where the scenery is uniformly distributed on $[0,1]$.

\section{Applications}

\subsection{Degenerate $U$-Statistics} There is a substantial amount of work for $U$-statistics indexed by a random walk, starting with Cabus and Guillotin-Plantard \cite{cab} for a degenerate $U$-statistic and a two-dimensional random walk. Also in the degenerate case, Guillotin-Plantard and Ladret \cite{guil} study one dimensional random walks with $\alpha>1$. Non-degenerate $U$-statistics are investigated by Franke, P\`ene and Wendler \cite{franke}. 
Theorem \ref{theo1} gives an alternative proof of Theorems 1.1 and 5.1 in \cite{cab} in the case of degenerate $U$-statistics indexed by a random walk if the kernel has bounded total variation. The arguments can be found in Dehling, Taqqu \cite{dehling}. For sake of completeness, we give the proof of the convergence of the one-dimensional distributions. Convergence of the finite-dimensional distributions and tightness are omitted.

Let $h:[0,1]\times[0,1] \rightarrow\R$ be a symmetric function with bounded total variation. We study the statistic
\begin{equation*}
U_n(h):=\frac{1}{n^2}\sum_{i,j=1}^n h(\xi_{S_i},\xi_{S_j})
\end{equation*}
where $(S_{n})_{n}$ is a $\Z^{d}$-random walk satisfying either (a) or (b2) from Theorem \ref{theo1} and $(\xi_{x})_{x\in\Z^{d}}$ is defined as in the introduction.
If  $h$ is degenerate, meaning that $\E h(x,\xi_1)=\int_0^1 h(x,y)\, dy=0$ for all $x\in\R$, we get the following expansion using the distribution function $F(s)=s$ and the empirical distribution function $F_n(s):=\frac{1}{n}\sum_{i=1}^n\mathds{1}_{\{\xi_{S_i}\leq s\}}$:
\begin{align*}
&U_n(h)-\E\left[h\big(\xi_1,\xi_2\big)\right]\\
=&\int_0^1\!\!\!\int_0^1h(x,y)dF_n(x)dF_n(y)-\int_0^1\!\!\!\int_0^1h(x,y)dF(x)dF(y)\\
=&\int_0^1\!\!\!\int_0^1\!\!h(x,y)d(F_n\!-\!F)(x)d(F_n\!-\!F)(y)+2\int_0^1\!\!\!\int_0^1\!\!h(x,y)dF(x)d(F_n\!-\!F)(y).
\end{align*}
The second integral equals 0 because of the degeneracy, and using integration by parts, we obtain
\begin{multline*}
U_n(h)-\E\left[h\big(\xi_1,\xi_2\big)\right]=\int_0^1\!\!\!\int_0^1\!(F_n\!-\!F)(x)(F_n\!-\!F)(y)dh(x,y)\\
=\frac{a_n^2}{n^2}\int_0^1\!\!\!\int_0^1\!a_n^{-1}W_n(x,1)a_n^{-1}W_n(y,1)dh(x,y).
\end{multline*}
So we conclude that the $U$-statistic converges in distribution.

\subsection{Testing for Stationarity of the Scenery}

There is a growing interest in change point analysis and there are various tests for the hypothesis of stationarity against the alternative of a change of the distribution of a time series.  While most of the test prespecify the type of change, e.g. a change in location or in scale, various authors have proposed more general change point tests, which can detect any possible change in the distribution function.

Carlstein \cite{carlstein} proposed different tests for change in distribution of independent random variables. A test under short range dependence was developed by Inoue \cite{inoue}. Giraitis, Leipus and Surgailis \cite{giraitis} and Tewes \cite{tewes} have studied this problem under long range dependence. In the long range dependent case, an interesting phenomenon can appear: The general test for a change in distribution can have the same asymptotic power under a change in mean as the classical CUSUM test, which is specialized to detect a shift in mean, see \cite{tewes}.
If the scenery is not stationary, the random walk in random scenery might be non-stationary. Especially in the transient case, if the distribution of the scenery is different in different regions, this should be observable, because the random walk will pass this different regions. Following Inoue \cite{inoue}, we propose the test statistic
\begin{equation*}
T_n:=\max_{1\leq k <n}\sup_{s\in\R}\bigg|\sum_{i=1}^k\mathds{1}_{\{X_{S_i}\leq s\}}-\frac{k}{n}\sum_{i=1}^n\mathds{1}_{\{X_{S_i}\leq s\}}\bigg|.
\end{equation*}
This statistic compares the empirical distribution function of the first $k$ observed values with the empirical distribution function of all observed values (taking the maximum over all $k\leq n$). Under the alternative, it is sensitive to a change of the distribution when the random walk goes to different regions. Under the hypothesis of a stationary scenery, we get the asymptotic distribution of the test statistic by using the continuous mapping theorem:
\begin{multline*}
a_n^{-1}T_n=\sup_{t\in[0,1]}\sup_{s\in\R}\Big|a_n^{-1}V_n(s,t)-a_n^{-1}tV_n(s,1)\Big|\\
\Rightarrow \sqrt{c}\sup_{t\in[0,1]}\sup_{s\in\R}\Big|W(F_X(s),t)-tW(F_X(s),1)\Big|
\end{multline*}
A continuous distribution function $F_X$ takes every value $x\in[0,1]$ by the intermediate value theorem. So in this case, we recognize that the supremum above is the supremum of the Brownian pillow $(W(s,t)-tW(s,1))_{s,t\in[0,1]}$.

\section{Proof}
\subsection{Recalls and auxiliary results}
We define the occupation times as $N_n(x):=\sum_{i=1}^n\mathds{1}_{\{S_i=x\}}$. 
Assume that the random walk satisfies one of the hypotheses of Theorem \ref{theo1}, then for any $\varepsilon >0$,
\begin{equation}\label{Nn*}
\sup_{x\in\Z^d}N_n(x)=o(n^{\varepsilon})\ \ \text{a.s.}
\end{equation}
(see the proof of Lemma 2.5 in \cite{bolt89}).\\ 
Moreover
\begin{equation}\label{sumNn^2}
\lim_{n\rightarrow \infty}\frac 1{a_n^2}\sum_{x\in\Z^d} N_{n}^2(x)=c\ \  \text{a.s.},
\end{equation}
where
\begin{itemize}
\item $c=1+2\sum_{n\ge 1}\P(S_n=0)$ in Case (a) (see the introduction of \cite{kest79}), 
\item $c=2/\pi A$ in Case (b) (see \cite{cab,cerny07, degli11}).
\end{itemize}
\begin{lemma}\label{lem0} Under the assumptions of Theorem \ref{theo1}, for every $a<b$,
\begin{equation}\label{inter}
\sum_{k=0}^{\lfloor an\rfloor}\sum_{l=\lfloor an\rfloor+1}^{\lfloor bn\rfloor} \mathds{1}_{\{S_k =S_l\}}=o\left( a_n^2\right)\, \quad a.s..
\end{equation}
\end{lemma}
\begin{proof} See Proposition 2.3 in \cite{cab}.
\end{proof}
As a consequence of \eqref{sumNn^2} and of Lemma \ref{lem0}, we obtain
\begin{equation}\label{sumNnst}
 \forall s,t>0, \quad  \lim_{n\rightarrow +\infty} a_n^{-2}\sum_{x\in\mathbb Z^d} N_{\lfloor nt\rfloor}(x)N_{\lfloor ns\rfloor}(x)=c\, \min(s,t)\quad a.s.\, .
\end{equation}
We will proceed with some moment bounds for the occupation times:

\begin{lemma}\label{lem1} Let $(S_n)_{n\in\N}$ be a transient random walk in $\Z^d$, then there exists some constant $C$, such that for all $n\geq 1$
\begin{align*}
\E\bigg[\sum_{x\in\Z^d}N_{n}^4(x)\bigg]&\leq C n.\\
\E\bigg[\Big(\sum_{x\in\Z^d}N_{n}^2(x)\Big)^2\bigg]&\leq C n^2.
\end{align*}
\end{lemma}
\begin{proof} 
We can follow the proof of item (i) of Proposition 2.3 in \cite{guil13} using the fact that, for all $k\in\N$, 
$\sup_n \E[ N_n(0)^k] = \E[N_{\infty}(0)^k]<+\infty$ since $N_{\infty}(0)$ has Geometric distribution with parameter $\P(S_n\neq 0\ \text{forall}\ n\geq 1)>0$ (see also Lemma 7 and 8 in \cite{GuiRen}).

\end{proof}

\begin{lemma}\label{lem2} Let $(S_n)_{n\in\N}$ be a recurrent random walk in $\Z^d$ such that $(n^{-\frac 1d}S_n)_n$ converges in distribution to a stable distribution of index $d\in\{1,2\}$, then for some constants $C_1,C_2\in(0,\infty)$
\begin{align*}
\E\bigg[\sum_{x\in\Z^d}N_{n}^4(x)\bigg]&\leq C_1 n \log^3(n).\\
\E\bigg[\Big(\sum_{x\in\Z^d}N_{n}^2(x)\Big)^2\bigg]&\leq C_2 n^2 \log^2(n).
\end{align*}
\end{lemma}
\begin{proof} See Proposition 2.3 in \cite{guil13}.
\end{proof}
As usual, our proof will be divided in two steps: we will prove the convergence of 
the finite-dimensional distributions in Section \ref{fdd} and, then, we will prove the tightness in Section \ref{tight}.

\subsection{Convergence of the finite-dimensional distributions}\label{fdd}

We introduce the following notation for $x\in\Z^d$ and $s\in[0,1]$:
$$\zeta_s(x):=\mathds 1_{\xi_x\leq s}-s.$$
We have
to prove for any $m,k\in\N^{*}$, $s_1,\ldots,s_k\in[0,1]$ and $t_1,\ldots,t_m\in[0,1]$ the convergence in distribution of 
\begin{equation*}
\Bigg(\frac 1{a_n}\Big(\sum_{x\in\Z^d} N_{\lfloor nt_j \rfloor}(x)\zeta_{s_i}(x)\Big)_{\substack{i=1,\ldots,k\\j=1,...,m}} \Bigg)_n
\end{equation*}
 to the random vector $(W(s_i,t_j))_{i=1,\ldots,k,\  j=1,...,m}$. Let us fix $s_1,\ldots,s_k\in[0,1]$, $t_1,\ldots,t_m\in [0,1]$ and $\theta_{i,j}\in\R$ for $i=1,\ldots,k$, $j=1,...,m$.
Let $\varphi_n$ denote the characteristic function of the previous vector and $\mathcal{F}$ the $\sigma-$field generated by the random walk. Using the independence between the random scenery and the random walk, a simple computation gives
\begin{align*}
&\varphi_n\big((\theta_{i,j})_{i=1,\ldots,k,j=1,\ldots,m}\big)\\
&= \E\bigg[ \prod_{x\in \Z^d}\E \bigg[\exp\Big(i\frac {1}{a_n}\sum_{i=1}^k\sum_{j=1}^m\theta_{i,j}N_{\lfloor nt_j\rfloor }(x)\zeta_{s_i}(x)\Big)\Big|  \mathcal{F}\bigg]\bigg]\displaybreak[0]\\
&= \E\bigg[ \prod_{x\in \Z^d}\varphi_{s_1,...,s_k}
 \Big(\frac 1{a_n}\sum_{j=1}^m\theta_{1,j} N_{\lfloor nt_j\rfloor }(x),\ldots,\frac 1{a_n} \sum_{j=1}^m \theta_{k,j} N_{\lfloor nt_j\rfloor }(x)\Big)\bigg],
\end{align*}
with $\varphi_{s_1,...,s_k}$ the characteristic function of $(\zeta_{s_1}(0),\ldots,\zeta_{s_k}(0))$. 
Denote by $U_n(x)$ the random vectors defined by 
$$ \frac 1{a_n}\Big(\sum_{j=1}^m\theta_{1,j}N_{\lfloor nt_j\rfloor }(x),\ldots,\sum_{j=1}^m \theta_{k,j} N_{\lfloor nt_j\rfloor }(x)\Big), \ x\in \Z^d$$
and $\Sigma=(\sigma_{i,i'})_{i,i'=1,...,k}$ the covariance matrix of  $(\zeta_{s_1}(0),...,\zeta_{s_k}(0))$.
We firstly prove that 
$$ \E\bigg[ \prod_{x\in \Z^d}\varphi_{s_1,...,s_k}
 \Big(U_n(x)\Big)\bigg] - \E\bigg[\prod_{x\in\Z^d} e^{-\frac {1}{2}\langle \Sigma U_n(x),U_n(x)\rangle} \bigg]\xrightarrow{n\rightarrow\infty} 0.$$
Note that the above products, although indexed by $x\in\mathbb{Z}^d$, have
only a
finite number of factors different from $1$.
And furthermore, all factors are complex numbers in $\bar{\mathbb{D}} =
\left
\{z\in\mathbb{C}\mid  |z|\leq1\right\}$.
We use the following inequality: Let $(z_i)_{i\in I}$ and $(z_i')_{i\in
I}$ be two families of complex numbers in $\bar{\mathbb{D}}$ such
that all
terms are equal to one, except a finite number of them. Then
\[
\Biggl|\prod_{i\in I}z'_i -\prod_{i\in I}z_i\Biggr| \leq\sum_{i\in I}
|z'_i - z_i|.
\]
This yields
\begin{align}\label{diff-car}
&\bigg| \E\bigg[ \prod_{x\in \Z^d}\varphi_{s_1,...,s_k}
 \Big(U_n(x)\Big)\bigg]   -\E\bigg[\prod_{x\in\Z^d} e^{-\frac {1}{2}\langle \Sigma U_n(x),U_n(x)\rangle} \bigg]\bigg|\nonumber\\
\leq& \sum_{x\in\Z^d}\E\bigg[\Big| \varphi_{s_1,...,s_k} \Big(U_n(x)\Big) - e^{-\frac {1}{2}\langle \Sigma U_n(x),U_n(x)\rangle} \Big|\bigg].
 \end{align}
Note that the random variables $\zeta_{s_1}(0),\ldots,\zeta_{s_k}(0)$ are bounded and therefore $\varphi_{s_1,...,s_k}(u)=e^{-\frac 12\langle \Sigma u,u\rangle}+o(|u|_\infty^2)$
as $u\rightarrow 0$,
with $|u|_\infty=\max\{|u_1|,\ldots,|u_k|\}$. We denote by $g$ the continuous and bounded function defined on $\R^k$ by $g(0)=0$ and $$g(u) =  |u|_{\infty}^{-2} \left|\varphi_{s_1,...,s_k} (u) - e^{-\frac 12\langle \Sigma u,u\rangle}\right|
\le 2/|u|_{\infty}^{2} 
$$
so that
$$\bigg| \varphi_{s_1,...,s_k} \Big(U_n(x)\Big) - e^{-\frac {1}{2}\langle \Sigma U_n(x),U_n(x)\rangle} \bigg| = \left|U_n(x)\right|_{\infty}^{2} g(U_n(x)).$$
Let us define $U_n =\max_{x\in\Z^d} \left|U_n(x)\right|_{\infty}$ and the function $\tilde{g}:[0,+\infty)\rightarrow [0,+\infty)$ by 
$\tilde{g}(u) =\sup_{|v|_{\infty}\leq u} |g(v)|.$
Note that $\tilde{g}$ is continuous, vanishes at 0
and is bounded.
Then, for any $x\in\Z^{d}$,
\begin{equation}\label{approx-car}
\bigg| \varphi_{s_1,...,s_k} \Big(U_n(x)\Big) - e^{-\frac {1}{2}\langle \Sigma U_n(x),U_n(x)\rangle} \bigg| \leq  \left|U_n(x)\right|_{\infty}^{2} \tilde{g}(U_n).
\end{equation}
Equations (\ref{diff-car}) and (\ref{approx-car}) together yield
\begin{align}\label{estim}
&\bigg| \E\bigg[ \prod_{x\in \Z^d}\varphi_{s_1,...,s_k}
 \Big(U_n(x)\Big)\bigg]   -\E\bigg[\prod_{x\in\Z^d} e^{-\frac {1}{2}\langle \Sigma U_n(x),U_n(x)\rangle} \bigg]\bigg|\nonumber\\
\leq& \E\bigg[\tilde{g}(U_n) \sum_{x\in\Z^d}|U_n(x)|_\infty^2 \bigg]\nonumber\\
\leq&m^2 \max_{i,j} \left|\theta_{i,j}\right|^2 \E\bigg[\tilde{g}(U_n) \bigg(\frac{1}{a_n^2} \sum_{x\in\Z^d} N_n(x)^2 \bigg)\bigg].
 \end{align}
Due to \eqref{Nn*}, $U_n$ converges almost surely to 0 as $n$ goes to infinity. Since $\tilde{g}$ is continuous and vanishes at 0, $\tilde{g}(U_n)$ converges almost surely to 0. 
Using \eqref{sumNn^2}, the second term in the expectation of \eqref{estim} converges almost surely to some constant. Moreover, from Lemma \ref{lem1} respectively Lemma \ref{lem2}, we 
know that this term is also bounded in $L_2$. Since $\tilde{g}$ is bounded, we can conclude that
\begin{equation*}
\varphi_n\big((\theta_{i,j})_{i,j}\big)-\E\Big[\prod_{x\in\Z^d} e^{-\frac {1}{2}\langle \Sigma U_n(x),U_n(x)\rangle}\Big]\xrightarrow{n\rightarrow\infty} 0.
\end{equation*}
Now, due to \eqref{sumNnst}, we obtain
\begin{eqnarray*}
\lim_{n\rightarrow \infty}\varphi_n((\theta_{i,j})_{i,j})&=&\exp\Big(-\frac c2\sum_{i,i'=1}^k\sum_{j,j'=1}^m \theta_{i,j}\theta_{i',j'} \ (t_j\wedge t_{j'}) \sigma_{i,i'}\Big)\\
&=& \E\bigg[ \exp\bigg(i \sum_{i=1}^k\sum_{j=1}^m \theta_{i,j}\sqrt{c}W(s_i,t_j) \bigg)\bigg]
\end{eqnarray*}
by remarking that $\sigma_{i,i'} = s_i \wedge s_{i'} -s_i s_{i'}$.

\subsection{Tightness}\label{tight}

The proof of the tightness follows in the same way as in \cite{wend16}.
Recall that we assume that $(\xi_x)_{x\in\Z^d}$ are uniformly distributed on the interval $[0,1]$. In this situation, we have for all $x\in\Z^d$ and $s_1,s_2\in[0,1]$
\begin{align*}
\E\left[\big(\zeta_{s_1}(x)-\zeta_{s_2}(x)\big)^2\right]&\leq |s_1-s_2|,\\
\E\left[\big(\zeta_{s_1}(x)-\zeta_{s_2}(x)\big)^4\right]&\leq |s_1-s_2|.
\end{align*}
Now, using inequalities from Lemmas \ref{lem1} and \ref{lem2}, we obtain the following moment bound for all $n_1< n_2\leq n$ and $s_1,s_2\in[0,1]$ with $|s_1-s_2|\geq 1/n$:§
\begin{align*}
 &\E\left[\bigg(a_n^{-1}\sum_{i=n_1+1}^{n_2}\big(\zeta_{s_1}(S_i)-\zeta_{s_2}(S_i)\big)\bigg)^4\right]\\
=&\E\left[\bigg(a_n^{-1}\sum_{x\in\Z^d}N_{n_2-n_1}(x)\big(\zeta_{s_1}(x)-\zeta_{s_2}(x)\big)\bigg)^4\right]\\
\leq & \E\bigg[\big(\zeta_{s_1}(0)-\zeta_{s_2}(0)\big)^4\bigg] \E\bigg[a_n^{-4} \sum_{x\in\Z^d} N_{n_2-n_1}^4(x)\bigg]\\
\ & +\E\bigg[\big(\zeta_{s_1}(0)-\zeta_{s_2}(0)\big)^2\bigg]^2 \E\bigg[\bigg(a_n^{-2} \sum_{x\in\Z^d} N_{n_2-n_1}^2(x)\bigg)^2\bigg] \\
\leq&\frac{C_1}{a_n^4}\left[a_{n_2-n_1}^2\log^2 (n_2-n_1)|s_1-s_2|+a_{n_2-n_1}^{4}(s_1-s_2)^2\right]\\
\leq&{C_1}\left[\frac{n_2-n_1}{n^2}\log (n_2-n_1)|s_1-s_2|+\left( \frac{n_2-n_1}n\right)^2(s_1-s_2)^2\right]\\
\leq&2 C_1\Big(\frac{n_2-n_1}{n}\Big)^{3/2}|s_1-s_2|^{3/2}.
\end{align*}
If $s_1<s_2$ and $|s_1-s_2|\leq 2/n$, we have by monotonicity that for any $s\in(s_1,s_2)$
\begin{align}\label{bickel}
 &\bigg|\frac{1}{a_n}\sum_{i=n_1+1}^{n_2}\zeta_{s}(S_i)-\frac{1}{a_n}\sum_{i=n_1+1}^{n_2}\zeta_{s_1}(S_i)\bigg|\nonumber\\
\leq& \bigg|\frac{1}{a_n}\sum_{i=n_1+1}^{n_2}\mathds{1}_{\{\xi_{S_i}\leq s\}}-\frac{1}{a_n}\sum_{i=n_1+1}^{n_2}\mathds{1}_{\{\xi_{S_i}\leq s_1\}}\bigg|+\frac{n_2-n_1}{a_n}|s-s_1|\displaybreak[0]\nonumber\\
\leq& \bigg|\frac{1}{a_n}\sum_{i=n_1+1}^{n_2}\mathds{1}_{\{\xi_{S_i}\leq s_2\}}-\frac{1}{a_n}\sum_{i=n_1+1}^{n_2}\mathds{1}_{\{\xi_{S_i}\leq s_1\}}\bigg|+\frac{n_2-n_1}{a_n}|s_2-s_1|\displaybreak[0]\nonumber\\
\leq&\bigg|\frac{1}{a_n}\sum_{i=n_1+1}^{n_2}\zeta_{s_2}(S_i)-\frac{1}{a_n}\sum_{i=n_1+1}^{n_2}\zeta_{s_1}(S_i)\bigg|+2\frac{n_2-n_1}{a_n}|s_2-s_1|\nonumber\\
\leq& \bigg|\frac{1}{a_n}\sum_{i=n_1+1}^{n_2}\zeta_{s_2}(S_i)-\frac{1}{a_n}\sum_{i=n_1+1}^{n_2}\zeta_{s_1}(S_i)\bigg|+\frac{4}{a_n}.
\end{align}
Following Bickel and Wichura \cite{bick71}, we introduce for a two-parameter stochastic process $(V(s,t))_{s,t\in[0,1]}$ the notation
\begin{multline*}
 w''_{\delta}(V)\\
 :=\max\Big\{\sup_{\substack{0\leq t_1\leq t\leq t_2\leq 1\\t_2-t_1\leq \delta}}\min\left\{\|V(\cdot,t_2)-V(\cdot,t)\|_\infty,\|V(\cdot,t)-V(\cdot,t_1)\|_\infty\right\},\\
\sup_{\substack{0\leq s_1\leq s\leq s_2\leq 1\\s_2-s_1\leq\delta}}\min\left\{\|V(s_2,\cdot)-V(s,\cdot)\|_\infty,\|V(s,\cdot)-V(s_1,\cdot)\|_\infty\right\}\Big\},
\end{multline*}
where $\|\cdot\|_\infty$ denotes the supremum norm. For $(W_n(s,t))_{s,t\in[0,1]^2}$ with
\begin{equation*}
W_n(s,t):=\frac{1}{a_n}\sum_{i=1}^{[nt]}\zeta_{s}(S_i)
\end{equation*}
and the index set $D_n:=\left\{0,\frac{1}{n},\frac{2}{n},\ldots,1\right\}^2$, we have by \eqref{bickel}
\begin{equation*}
w''_{\delta}(W_n)\leq w''_{\delta}(W_{n|D_n})+\frac{4}{a_n},
\end{equation*}
where $ w''_{\delta}(W_{n|D_n})$ is calculated by restricting all suprema to the set $D_n$. Now by Theorem 3 (and the remarks following their theorem) of Bickel and Wichura \cite{bick71} together with our moment bound
\begin{equation*}
 \E\left[\bigg(\frac{1}{a_n}\sum_{i=n_1+1}^{n_2}\big(\zeta_{s_1}(S_i)-\zeta_{s_2}(S_i)\big)\bigg)^4
\right]\leq 2 C_1\Big(\frac{n_2-n_1}{n}\Big)^{3/2}|s_1-s_2|^{3/2},
\end{equation*}
we can conclude that for any $\epsilon>0$
\begin{equation*}
 \P\left(\limsup_{n\rightarrow\infty}w''_{\delta}(W_{n|D_n})>\epsilon\right)\xrightarrow{\delta\rightarrow 0}0
\end{equation*}
and consequently
\begin{equation*}
 \P\left(\limsup_{n\rightarrow\infty}w''_{\delta}(W_{n})>\epsilon\right)\xrightarrow{\delta\rightarrow 0}0.
\end{equation*}
Thus the process is tight by Corollary 1 of \cite{bick71}.

\end{document}